\documentclass[11pt]{article}

\usepackage{arxiv}
\usepackage{natbib}

\usepackage{amsmath,amsfonts,bm}

\def\eqref#1{equation~\ref{#1}}

\def\1{\bm{1}}

\DeclareMathAlphabet{\mathsfit}{\encodingdefault}{\sfdefault}{m}{sl}
\SetMathAlphabet{\mathsfit}{bold}{\encodingdefault}{\sfdefault}{bx}{n}

\def\sR{{\mathbb{R}}}

\usepackage{url}
\usepackage{xcolor}
\usepackage{tikz}
\usetikzlibrary{patterns}

\usetikzlibrary{arrows.meta}

\usepackage{comment}
\usepackage[page]{appendix}
\RequirePackage[american]{babel}
\RequirePackage{csquotes}
\RequirePackage[T1]{fontenc}

\usepackage{microtype}
\usepackage{graphicx}
\usepackage{subcaption}
\usepackage{booktabs} 

\usepackage{hyperref}
\hypersetup{
  colorlinks=true,
  citecolor=blue,
  urlcolor=cyan,
  linkcolor=magenta,
}

\usepackage{amsmath}
\usepackage{amssymb}
\usepackage{mathtools}
\usepackage{amsthm}

\usepackage{algorithm}
\usepackage{algpseudocode} 

\usepackage{multirow}
\usepackage{array}
\usepackage{comment}
\usepackage{enumitem}

\newtheorem{problem}{Problem}

\newcommand{\CBGC}{\hyperref[prob:cbgc]{\textsc{CBGC}}}
\newcommand{\GQGC}{\hyperref[prob:gqgc]{\textsc{GQGC}}}

\newcommand{\GC}{%
  \hyperref[prob:gap-convexity]{%
    \textsc{Gap-Convexity}}}

\newcommand{\SC}{%
  \hyperref[prob:signomial-convexity]{%
    \textsc{Signomial-Convexity}}}

\newcommand{\LCC}{%
  \hyperref[prob:log-coordinate-convexity]{%
    \textsc{Log-Coordinate-Convexity}}}

\newcommand{\LLC}{%
  \hyperref[prob:log-log-convexity]{%
    \textsc{Log-Log-Convexity}}}

\newcounter{problemvariant}[problem]

\newlength{\problempairsep}
\newtheorem{theorem}{Theorem}[section]

\newtheorem{lemma}[theorem]{Lemma}

\usepackage[textsize=tiny]{todonotes}

\usepackage{adjustbox}
\usepackage{mathrsfs}
\usepackage{stmaryrd}
\usepackage{bm}

\RequirePackage{booktabs}

\RequirePackage{siunitx}
\renewcommand{\headeright}{}

\newcommand\blfootnote[1]{%
  \begingroup
  \renewcommand\thefootnote{}\footnote{#1}%
  \addtocounter{footnote}{-1}%
  \endgroup
}

\usepackage[capitalize,noabbrev]{cleveref}

\title{Recognizing Signomial Convexity is Hard}

\usepackage{footmisc}
\usepackage{authblk}

\author[1]{Rui Zheng}
\author[1]{Iosif Sakos}
\author[1,2,3]{Antonios Varvitsiotis}
\affil[1]{Singapore University of Technology and Design, Singapore}
\affil[2]{Centre for Quantum Technologies, National University of Singapore, Singapore}
\affil[3]{Archimedes/Athena Research Center, Greece}
\date{\vspace{-8ex}}

\begin{document}

\maketitle
\blfootnote{Contact: \texttt{rui\_zheng@mymail.sutd.edu.sg}, \texttt{iosif\_sakos@sutd.edu.sg}, \texttt{antonios@sutd.edu.sg}}

\begin{abstract}
Signomials---finite sums of generalized monomials with real exponents---are a natural machine learning model, combining parsimonious representations of power laws, inverse relationships, and multiplicative interactions with universal approximation and interpretable parameters. In fact, 45 of the 100 equations in the AI Feynman benchmark admit signomial representations. Moreover, signomial optimization is widely used in engineering design, communications, economics, and machine learning, but is computationally intractable in general. Convexity, the gold standard for efficient optimization, offers three routes to tractability for signomials: a signomial may be convex in its original variables, become convex after a logarithmic change of variables, or, when positive, become convex after additionally taking the logarithm of its value---the structure underlying disciplined geometric programming. We show that recognizing each form is strongly NP-hard, both on compact domains and globally. Our proofs start from gap-promise variants of polynomial convexity and develop polynomial-to-signomial reductions that preserve curvature gaps under the relevant logarithmic transformations.
\end{abstract}

\section{Introduction}
\label{sec:introduction}

Signomials are finite linear combinations of generalized monomials with real exponents.
Formally, a signomial in $n$ variables is a function 
\begin{equation}
\label{eq:signomial_in_x}
        f(x)
    =
    \sum_{\alpha\in A}
    c_{\alpha}
    \prod_{j=1}^{n}x_j^{\alpha_j},
    \qquad
    x_1,\ldots,x_n>0,
\end{equation}
where $A\subset\sR^n$ is finite and $c_{\alpha}\in\sR$. In particular, signomials substantially generalize polynomial models by allowing fractional and negative powers: expressions such as \(\sqrt{x_1}\), \(x_2^{-1}=1/x_2\), and \(x_1^{3/2}x_2^{-1/3}\) are all generalized monomials.
Because arbitrary real powers need not be real-valued at nonpositive arguments, the natural domain of a signomial is the positive orthant $\sR_{++}^n$.

Signomials are a natural model class for scientific equation discovery. The AI Feynman benchmark contains \(100\) equations drawn from the three volumes of \emph{The Feynman Lectures on Physics}~\citep{udrescu2020ai},  45 of which admit signomial representations~\citep{lumadjeng2026ecsel}. This prevalence reflects the ability of signomials to represent power laws, inverse relationships, roots, and multiplicative interactions exactly. For machine-learning applications, signomials combine parsimony, interpretability, and expressive power~\citep{karcher2023data,lee2014sparse,hwang2015multi,lumadjeng2026ecsel}. First, relationships such as \(\sqrt{x}\) and \(1/x\) require a single signomial term, whereas polynomials and neural networks can approximate them but do not explicitly encode their power-law structure. Second, the model parameters (the exponents) have direct interpretations as elasticities. Third, signomials inherit the universal-approximation property of polynomials on compact subsets of \(\sR_{++}^n\).

The problem of minimizing a signomial \(f(x)\) over \(x\in\sR_{++}^n\), subject to signomial constraints \(f_i(x)\leq 1\) and \(g_i(x)=1\), is called a \emph{signomial optimization problem}~\citep[Section~9.1]{boyd2007tutorial}. Such problems arise in machine learning, wireless communications, inventory management, biotechnology, aerospace engineering, and chemical engineering~\citep{karcher2023data,fevotte2011algorithms,zimmert2021tsallis,chiang2007power,chiu2007optimal,san2022inventory,marin2007optimization,kirschen2018application,bjork2002global,dressler2022algebraic,maranas1997global,lee2014sparse,hwang2015multi,lumadjeng2026ecsel}.

This expressiveness, however, comes with a computational cost: signomial optimization is NP-hard in general~\citep[see, e.g.,][]{chandrasekaran2016relative}. This follows already from matrix copositivity: deciding whether \(x^\top Qx\geq 0\) for every \(x\geq 0\) is coNP-complete~\citep{murty1987some}. Consequently, deciding whether the quadratic signomial \(x^\top Qx\) attains a negative value is also hard.

The hardness of signomial optimization motivates the search for tractable subclasses, foremost among them convex optimization problems.
For signomials, however, convexity can arise in three distinct forms.
First, $f$ may be convex in its original positive variables $x$.
Second, after the logarithmic reparameterization $x=e^y$, the function $y\mapsto f(e^y)$ may be convex on $\sR^n$; we call this \emph{log-coordinate convexity}.
Third, when $f>0$, the function $y\mapsto\log f(e^y)$ may be convex; this property is known as \emph{log-log convexity}.
Figure~\ref{fig:convexity-notions} summarizes these three notions and their relationships.

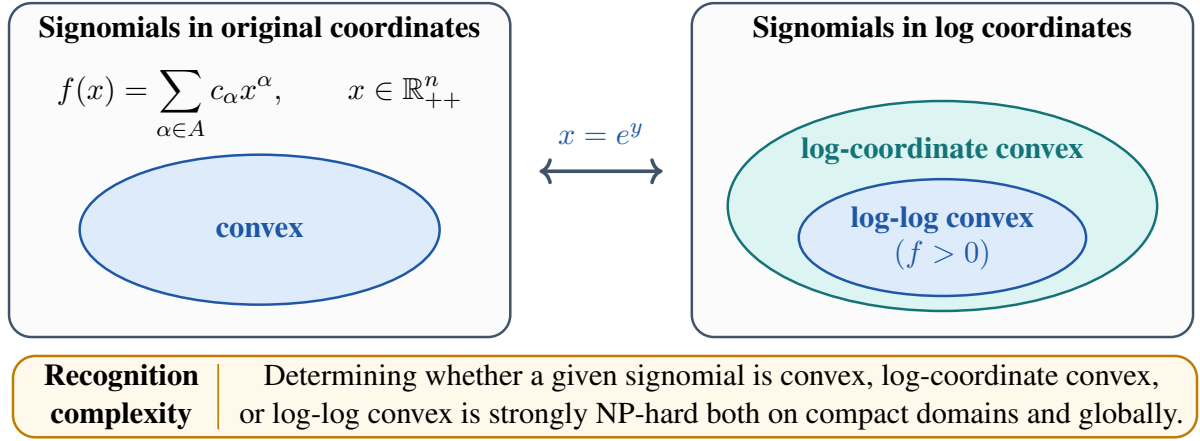
\begin{figure}[!ht]
\centering
\resizebox{\textwidth}{!}{%
\begin{tikzpicture}[x=1cm,y=1cm]

\definecolor{panelborder}{RGB}{63,83,105}
\definecolor{blueedge}{RGB}{30,88,165}
\definecolor{bluefill}{RGB}{222,235,250}
\definecolor{tealedge}{RGB}{20,115,120}
\definecolor{tealfill}{RGB}{220,244,242}
\definecolor{goldedge}{RGB}{190,125,0}
\definecolor{goldfill}{RGB}{255,249,235}


\draw[
    rounded corners=8pt,
    line width=0.9pt,
    draw=panelborder,
    fill=gray!3
]
(0.00,1.92) rectangle (6.35,6.15);

\draw[
    rounded corners=8pt,
    line width=0.9pt,
    draw=panelborder,
    fill=gray!3
]
(8.65,1.92) rectangle (15.00,6.15);


\node[
    font=\bfseries\normalsize,
    align=center
] at (3.175,5.82)
    {Signomials in original coordinates};

\node[
    font=\normalsize,
    align=center
] at (3.175,4.88)
    {$\displaystyle
      f(x)=\sum_{\alpha\in A}c_\alpha x^\alpha,
      \qquad x\in\mathbb{R}_{++}^{n}$};

\draw[
    line width=0.9pt,
    draw=blueedge,
    fill=bluefill
]
(3.175,3.28) ellipse (2.28 and 0.95);

\node[
    font=\bfseries\normalsize,
    text=blueedge
] at (3.175,3.28)
    {convex};


\node[
    font=\normalsize,
    text=blueedge
] at (7.50,4.50)
    {$x=e^y$};

\draw[
    <->,
    line width=1.3pt,
    draw=panelborder
]
(6.72,4.02) -- (8.28,4.02);


\node[
    font=\bfseries\normalsize,
    align=center
] at (11.825,5.82)
    {Signomials in log coordinates};

\draw[
    line width=0.9pt,
    draw=tealedge,
    fill=tealfill
]
(11.825,3.58) ellipse (2.72 and 1.33);

\node[
    font=\bfseries\normalsize,
    text=tealedge
] at (11.825,4.28)
    {log-coordinate convex};

\draw[
    line width=0.9pt,
    draw=blueedge,
    fill=bluefill
]
(11.825,3.18) ellipse (1.82 and 0.74);

\node[
    align=center,
    font=\bfseries\normalsize,
    text=blueedge
] at (11.825,3.18)
    {log-log convex\\[-1pt]
     \normalfont\itshape $(f>0)$};

\draw[
    rounded corners=8pt,
    line width=0.9pt,
    draw=goldedge,
    fill=goldfill
]
(0.05,0.65) rectangle (15.05,1.68);

\draw[
    line width=0.6pt,
    draw=goldedge
]
(2.65,0.75) -- (2.65,1.58);

\node[
    align=center,
    font=\bfseries\normalsize
] at (1.425,1.165)
    {Recognition\\complexity};

\node[
    align=center,
    font=\normalsize
] at (8.90,1.165)
    {Determining whether a given signomial is convex, log-coordinate convex,\\
    or log-log convex is strongly NP-hard both on compact domains and globally.};

\end{tikzpicture}%
}
\caption{
The three notions of convexity for signomials and their recognition complexity.
}
\label{fig:convexity-notions}
\end{figure}

Log-log convexity plays a central role in disciplined geometric programming (DGP)~\citep{agrawal2019disciplined}, much as ordinary convexity does in disciplined convex programming~\citep{grant2006disciplined}.
DGP begins with a library of atoms whose log-log curvature---affine, convex, or concave---is known in advance.
It then uses curvature-preserving composition rules, guaranteeing that every compliant problem becomes convex after a logarithmic change of variables.
Classical geometric programming is a canonical special case: applying the logarithmic transformation to a posynomial---a signomial with positive coefficients---produces a log-sum-exp function~\citep{boyd2007tutorial}.

However, signomial models learned from data need not satisfy the curvature conditions enforced by these composition rules.
Regression and equation-learning methods estimate their coefficients and exponents directly from observations~\citep{karcher2023data,lee2014sparse,lumadjeng2026ecsel}, so the resulting models carry no a priori guarantee of convexity in any of the three senses above. This leads to the central question of this paper: \emph{Can we efficiently determine whether a given signomial is convex, log-coordinate convex, or log-log convex?}

\paragraph{Contributions.}
We give a strongly negative answer to our central question: recognizing each of the three forms of convexity is strongly NP-hard, both on compact domains and globally; see Figure~\ref{fig:convexity-notions}. These statements correspond to six distinct recognition problems. In the original variables \(x\), compact-domain hardness concerns a box contained in \(\sR_{++}^n\), whereas global hardness concerns the entire positive orthant. For log-coordinate and log-log convexity, the transformed variable \(y=\log x\) belongs to \(\sR^n\), so the corresponding domains are a compact box in \(\sR^n\) and all of \(\sR^n\), respectively. Importantly, compact-domain hardness does not immediately imply global hardness: a function that is convex on a box need not remain convex outside it, so the positive instances of the compact problem are not necessarily positive instances of the global problem.

Our reductions are from strongly NP-hard gap-promise problems for polynomial convexity derived from the results of \cite{ahmadi2013np,ahmadi2020complexity}. Hardness in the original variables follows easily from the polynomial case: an affine change of variables maps the relevant polynomial instances to signomials over the positive orthant (Theorem~\ref{thm:original-coordinate-hardness}). However, hardness for log-coordinate and log-log convexity is more delicate: exponential reparameterization alters curvature, while the additional logarithm in log-log convexity introduces a gradient-dependent term. Our main technical contribution is a family of polynomial-to-signomial reductions that preserves the curvature gap under these transformations.

Section~\ref{sec:convexity-problems-and-gap-reductions} develops the gap-preserving reduction framework used to control these curvature changes. Section~\ref{sec:compact-hardness} applies it to prove strong NP-hardness of recognizing log-coordinate convexity~(Theorem~\ref{thm:compact-log-coordinate}) and log-log convexity~(Theorem~\ref{thm:compact-log-log}) on the unit box. Finally, Section~\ref{sec:unconstrained-hardness} establishes the corresponding global results over \(\sR^n\)~(Theorems~\ref{thm:global-log-coordinate} and~\ref{thm:global-log-log-convexity-hardness}).

\paragraph{Related work.}
To the best of our knowledge, no prior work has studied the computational complexity of recognizing convexity, log-coordinate convexity, or log-log convexity of signomials. The relevant hardness results concern polynomial convexity. \cite{ahmadi2013np} proved that deciding whether a quartic polynomial is globally convex is strongly NP-hard, while \cite{ahmadi2020complexity} established strong NP-hardness for cubic polynomials over a box. These results underlie the global quartic and cubic box gap-convexity problems used in our reductions and yield direct reductions to signomial convexity in the original variables. They do not, however, settle the complexity of log-coordinate or log-log convexity: exponential reparameterization changes curvature, and the additional outer logarithm in the log-log setting introduces a gradient-dependent term.

Related work on signomial optimization has largely focused on identifying tractable subclasses or constructing global relaxations. Disciplined geometric programming provides a compositional framework for constructing problems whose convexity after a logarithmic transformation is certified in advance~\citep{boyd2007tutorial,agrawal2019disciplined}.  A complementary line of work develops global optimization methods based on SAGE certificates and relative-entropy relaxations~\citep{chandrasekaran2016relative,murray2021newton,murray2021signomial,wang2020positivstellensatz,dressler2022algebraic}. These methods certify signomial nonnegativity and construct hierarchies of convex relaxations for approximating global optima, but do not address recognition of the three convexity notions studied here.

\section{Convexity Problems and Reduction Framework}
\label{sec:convexity-problems-and-gap-reductions}

\subsection{Gap Convexity: definition and hardness}
\label{subsec:gap-convexity}

Let \(\Omega\subseteq\mathbb{R}^n\) be a convex set given by a finite
description, and let \(\mathcal{F}(\Omega)\) be a finitely encoded class of
twice continuously differentiable functions \(p:\Omega\to\mathbb{R}\).
For \(p\in\mathcal{F}(\Omega)\), define
\[
    \kappa_\Omega(p)
    =
    \inf_{\xi\in \Omega}
    \lambda_{\min}\!\bigl(\nabla^2 p(\xi)\bigr).
\]

\begin{problem}[\textsc{Gap-Convexity}]
\label{prob:gap-convexity}
The input consists of a function \(p\in\mathcal{F}(\Omega)\) and two positive
rational thresholds \(\gamma_+\) and \(\gamma_-\). It is promised that
exactly one of the following alternatives holds:
\[
\begin{aligned}
    \textsc{Yes:}\quad&
    \kappa_\Omega(p)\geq\gamma_+,\\
    \textsc{No:}\quad&
    \kappa_\Omega(p)\leq-\gamma_-.
\end{aligned}
\]
The task is to decide which alternative holds.
\end{problem}

The positive alternative implies that \(p\) is uniformly
\(\gamma_+\)-strongly convex on \(\Omega\), whereas the negative
alternative guarantees negative curvature and hence that \(p\) is not convex on \(\Omega\).

Given two gap-convexity problems, a reduction from the first to the second is a polynomial-time algorithm that maps an instance \(p\) of the first problem to an instance \(q\) of the second, so that solving the latter determines which alternative holds for \(p\). For such a reduction, suppose that
\[
    \bigl|\kappa_\Omega(q)-\kappa_\Omega(p)\bigr|\leq E.
\]
If \(p\) satisfies the positive alternative
\(\kappa_\Omega(p)\geq\gamma_+\), then
\(\kappa_\Omega(q)\geq\gamma_+-E\). Similarly, if \(p\) satisfies the
negative alternative \(\kappa_\Omega(p)\leq-\gamma_-\), then
\(\kappa_\Omega(q)\leq-\gamma_-+E\). Figure~\ref{fig:gap-preserving-reduction}
illustrates this mechanism. To preserve the sign of the curvature in
both alternatives, we require
    $E<\min\{\gamma_+,\gamma_-\}$.
In particular, choosing
    $E\leq\frac{1}{2}\min\{\gamma_+,\gamma_-\}$
retains curvature margins of at least \(\gamma_+/2\) and
\(\gamma_-/2\), respectively. The factor \(1/2\) is only for
convenience; any fixed fraction below one would suffice.

\begin{figure}[h]
\centering

\begin{tikzpicture}[x=1.05cm,y=1cm]

\def\xL{-2.25}
\def\xR{2.25}
\def\xp{1.55}     
\def\xq{1.05}     
\def\hh{0.10}

\begin{scope}
    \node[above=8pt] at (0,0.25)
        {\(\kappa_\Omega(p)\geq\gamma_+\)};

    \draw[->,thick]
        (\xL,0) -- (\xR+0.15,0);

    \draw[densely dashed,gray]
        (0,-0.25) -- (0,0.25);
    \node[below=3pt] at (0,-0.25) {\(0\)};

    \fill[blue!18!gray!25]
        (\xp,-\hh) rectangle (\xR,\hh);
    \draw[thick,blue!50!black]
        (\xp,0) -- (\xR,0);

    \draw[dashed,gray!70!black]
        (\xp,-0.28) -- (\xp,0.28);
    \fill[blue!50!black]
        (\xp,0) circle (1.7pt);

    \node[below=3pt] at (\xp,-0.28)
        {\(\gamma_+\)};
\end{scope}

\draw[
    -{Stealth[length=3.2mm,width=2.3mm]},
    line width=1.4pt,
    blue!55!black,
    double=teal!25,
    double distance=1.1pt
]
    (2.55,0.03)
    to[bend left=20]
    node[
        midway,
        above=5pt,
        fill=white,
        rounded corners=2pt,
        inner sep=2.5pt,
        text=black
    ] {\scriptsize reduction}
    (3.75,0.03);

\begin{scope}[shift={(6.3,0)}]
    \node[above=8pt] at (0,0.25)
        {\(\kappa_\Omega(q)\geq\gamma_+-E\)};

    \draw[->,thick]
        (\xL,0) -- (\xR+0.15,0);

    \draw[densely dashed,gray]
        (0,-0.25) -- (0,0.25);
    \node[below=3pt] at (0,-0.25) {\(0\)};

    \fill[teal!18]
        (\xq,-\hh) rectangle (\xR,\hh);
    \fill[
        pattern=north east lines,
        pattern color=teal!70!black
    ]
        (\xq,-\hh) rectangle (\xR,\hh);

    \draw[thick,teal!60!black]
        (\xq,0) -- (\xR,0);

    \draw[densely dotted,thick,teal!70!black]
        (\xq,-0.28) -- (\xq,0.28);
    \fill[teal!60!black]
        (\xq,0) circle (1.7pt);

    \node[below=3pt] at (\xq,-0.28)
        {\(\gamma_+-E\)};
\end{scope}

\end{tikzpicture}

\vspace{1.5em}

\begin{tikzpicture}[x=1.05cm,y=1cm]

\def\xL{-2.25}
\def\xR{2.25}
\def\xp{-1.55}    
\def\xq{-1.05}    
\def\hh{0.10}

\begin{scope}
    \node[above=8pt] at (0,0.25)
        {\(\kappa_\Omega(p)\leq-\gamma_-\)};

    \draw[->,thick]
        (\xL,0) -- (\xR+0.15,0);

    \draw[densely dashed,gray]
        (0,-0.25) -- (0,0.25);
    \node[below=3pt] at (0,-0.25) {\(0\)};

    \fill[blue!18!gray!25]
        (\xL,-\hh) rectangle (\xp,\hh);
    \draw[thick,blue!50!black]
        (\xL,0) -- (\xp,0);

    \draw[dashed,gray!70!black]
        (\xp,-0.28) -- (\xp,0.28);
    \fill[blue!50!black]
        (\xp,0) circle (1.7pt);

    \node[below=3pt] at (\xp,-0.28)
        {\(-\gamma_-\)};
\end{scope}

\draw[
    -{Stealth[length=3.2mm,width=2.3mm]},
    line width=1.4pt,
    blue!55!black,
    double=teal!25,
    double distance=1.1pt
]
    (2.55,0.03)
    to[bend left=20]
    node[
        midway,
        above=5pt,
        fill=white,
        rounded corners=2pt,
        inner sep=2.5pt,
        text=black
    ] {\scriptsize reduction}
    (3.75,0.03);

\begin{scope}[shift={(6.3,0)}]
    \node[above=8pt] at (0,0.25)
        {\(\kappa_\Omega(q)\leq-\gamma_-+E\)};

    \draw[->,thick]
        (\xL,0) -- (\xR+0.15,0);

    \draw[densely dashed,gray]
        (0,-0.25) -- (0,0.25);
    \node[below=3pt] at (0,-0.25) {\(0\)};

    \fill[teal!18]
        (\xL,-\hh) rectangle (\xq,\hh);
    \fill[
        pattern=north east lines,
        pattern color=teal!70!black
    ]
        (\xL,-\hh) rectangle (\xq,\hh);

    \draw[thick,teal!60!black]
        (\xL,0) -- (\xq,0);

    \draw[densely dotted,thick,teal!70!black]
        (\xq,-0.28) -- (\xq,0.28);
    \fill[teal!60!black]
        (\xq,0) circle (1.7pt);

    \node[below=3pt] at (\xq,-0.28)
        {\(-\gamma_-+E\)};
\end{scope}

\end{tikzpicture}

\caption{
In the positive alternative, the lower curvature bound decreases by
at most \(E\). In the negative alternative, the negative-curvature
threshold increases by at most \(E\).}
\label{fig:gap-preserving-reduction}
\end{figure}
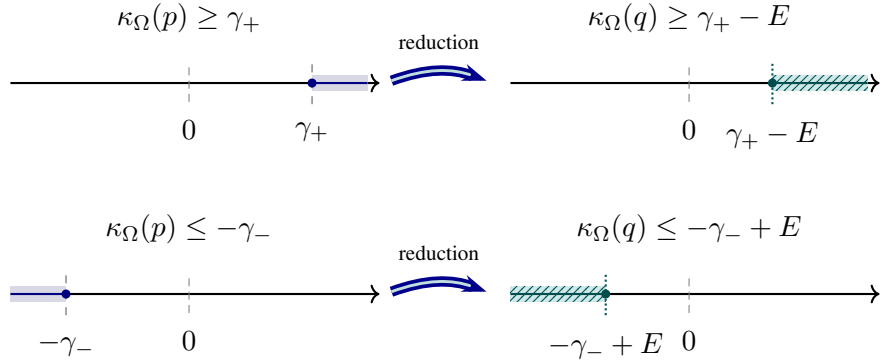

We base our reductions on the following two restrictions of
\GC{}:

\begin{description}
    \item[(\textsc{CBGC})]\phantomsection\label{prob:cbgc}
    \emph{Cubic box gap convexity:}
    \(\Omega=[-1,1]^n\) and
    \(p:\mathbb{R}^n\to\mathbb{R}\) is a cubic polynomial.

    \item[(\textsc{GQGC})]\phantomsection\label{prob:gqgc}
    \emph{Global quartic gap convexity:}
    \(\Omega=\mathbb{R}^n\) and
    \(p:\mathbb{R}^n\to\mathbb{R}\) is a quartic polynomial.
\end{description}

Both \CBGC{} and \GQGC{} are strongly NP-hard.
\cite{ahmadi2020complexity} established hardness of \CBGC{}
with inverse-polynomial curvature margins:
$
    \gamma_+,\gamma_-
    \geq \frac{1}{\operatorname{poly}(n)}.
$
For \GQGC{}, \cite{ahmadi2013np} established hardness with
the one-sided gap
$
    \gamma_+=0,
    \gamma_-\geq\frac{1}{\operatorname{poly}(n)}.$
In Appendix~\ref{app:quartic-gap-proof}, we modify their construction
to obtain $
    \gamma_+,\gamma_-
    \geq\frac{1}{\operatorname{poly}(n)},
$
as required by our gap-preserving reductions. 

In both cases, the hardness persists even when every coefficient
\(a/b\) of \(p\), written in lowest terms with \(b>0\), satisfies $|a|,b\leq\operatorname{poly}(n)$.
This is the meaning of the ``strong'' qualification in our hardness
statements: the coefficient numerators and denominators are
polynomially bounded, and the curvature margins are
inverse-polynomial in \(n\). Thus, the hardness does not rely on
exponentially large coefficients or exponentially small curvature
gaps.

\subsection{Signomial convexity-recognition problems}
\label{subsec:signomial-convexity-recognition-problems}

We study three signomial convexity-recognition problems, each with a constrained and an unconstrained variant. Throughout, \(f\) is a
signomial as in~\autoref{eq:signomial_in_x}, with rational coefficients
and rational exponents encoded in binary. Whenever \(f(e^y)\) appears,
exponentiation is applied componentwise.
In the constrained variants,
the input additionally includes a finitely described compact convex set \(X\subseteq\mathbb{R}_{++}^n\) or
\(Y\subseteq\mathbb{R}^n\), as appropriate.

\begin{problem}[\textsc{Signomial-Convexity}]
\label{prob:signomial-convexity}
Given a signomial \(f\), consider the following two variants:
\begin{enumerate}
\item[\textup{(i)}] \emph{Constrained:} Decide whether \(f\) is convex on \(X\).
\item[\textup{(ii)}] \emph{Unconstrained:} Decide whether \(f\)
is convex on \(\mathbb{R}_{++}^n\).
\end{enumerate}
\end{problem}

\begin{problem}[\textsc{Log-Coordinate-Convexity}]
\label{prob:log-coordinate-convexity}
Given a signomial \(f\), consider the following two variants:
\begin{enumerate}
\item[\textup{(i)}] \emph{Constrained:} Decide whether
\(y\mapsto f(e^y)\) is convex on \(Y\).
\item[\textup{(ii)}] \emph{Unconstrained:} Decide whether
\(y\mapsto f(e^y)\) is convex on \(\mathbb{R}^n\).
\end{enumerate}
\end{problem}

\begin{problem}[\textsc{Log-Log-Convexity}]
\label{prob:log-log-convexity}
Given a signomial \(f\), under the promise that \(f>0\)
throughout the relevant domain, consider the following two variants:
\begin{enumerate}
\item[\textup{(i)}] \emph{Constrained:} Decide whether
\(y\mapsto\log f(e^y)\) is convex on \(Y\).
\item[\textup{(ii)}] \emph{Unconstrained:} Decide whether
\(y\mapsto\log f(e^y)\) is convex on \(\mathbb{R}^n\).
\end{enumerate}
\end{problem}

Define \(g(y)=f(e^y)\). Accordingly, \LCC{} asks whether \(g\) is convex, whereas, under the promise that \(g>0\), \LLC{} asks whether \(\log g\) is convex.
Thus, \(f\) is log-log convex precisely when its exponential form
\(g\) is log-convex.

The two \textsc{Signomial-Convexity} problems admit direct reductions
from \CBGC{} and \GQGC{}. The constrained reduction uses an affine
bijection from a positive box onto the unit box, while the
unconstrained reduction uses the linear map \((u,v)\mapsto u-v\) from
\(\mathbb{R}_{++}^{2n}\) onto \(\mathbb{R}^n\). Both substitutions
produce ordinary polynomials, and hence signomials on the positive
orthant, so no approximation construction is needed.

\begin{theorem}
\label{thm:original-coordinate-hardness}
Both the constrained and unconstrained \SC{} problems are strongly NP-hard.
\end{theorem}

\begin{proof}
For the constrained problem, let \(p:\mathbb{R}^n\to\mathbb{R}\) be
an instance of \CBGC{} and define
$X=[1,2]^n$,
$T(x)=2x-3\mathbf{1}$,
$f(x)=p(T(x))$.
The affine map \(T\) sends \(X\) bijectively onto \([-1,1]^n\), and
$\nabla^2 f(x)=4\nabla^2p(T(x))$.
Consequently, \(f\) is convex on \(X\) if and only if \(p\) is convex
on \([-1,1]^n\). Moreover, \(f\) is a polynomial and hence a signomial
on \(\mathbb{R}_{++}^n\).

For the unconstrained problem, let
\(p:\mathbb{R}^n\to\mathbb{R}\) be an instance of \GQGC{} and define
$f(u,v)=p(u-v)$,
$(u,v)\in\mathbb{R}_{++}^{2n}$.
If \(p\) is convex on \(\mathbb{R}^n\), then \(f\) is convex as its
composition with a linear map. Conversely, suppose that \(f\) is
convex on \(\mathbb{R}_{++}^{2n}\). For any \(x,d\in\mathbb{R}^n\),
choose \(u,v\in\mathbb{R}_{++}^n\) such that \(u-v=x\). 
Let $z=(d^\top,0)^\top$, then
$z^\top
\nabla^2f(u,v)
z
=
d^\top\nabla^2p(x)d
\geq0$.
Since \(x\) and \(d\) are arbitrary, \(p\) is convex on
\(\mathbb{R}^n\). Thus, \(f\) is convex on
\(\mathbb{R}_{++}^{2n}\) if and only if \(p\) is convex on
\(\mathbb{R}^n\).

Both transformations are computable in polynomial time, preserve
fixed degree, and produce coefficients of polynomial encoding length.
Strong NP-hardness therefore follows.
\end{proof}

Unlike \textsc{Signomial-Convexity}, neither transformed problem inherits hardness directly from polynomial convexity. The exponential change of variables does not preserve curvature; for log-log convexity, taking the logarithm further changes the curvature condition and requires positivity.

The preceding discussion in
Section~\ref{subsec:gap-convexity} motivates a
\emph{gap-preserving reduction}. Given an instance $p$ of a known
hard gap-convexity problem, we seek to construct in polynomial time a function $q$ for the target recognition problem such that
    $|\kappa_Y(q)-\kappa_Y(p)|
    \leq E$,
    $E\leq\frac{1}{2}\min\{\gamma_+,\gamma_-\}$.
This bound preserves positive curvature in one alternative and
negative curvature in the other. Thus, an algorithm for the target recognition problem would solve the known hard gap-convexity problem.
If the construction also preserves the required encoding bounds,
strong NP-hardness transfers as well.
By Weyl's inequality~\citep{bhatia2013matrix}, it suffices to establish
\begin{equation}
\label{eq:hessian-perturbation}
    \sup_{\xi\in Y}
    \|\nabla^2q(\xi)-\nabla^2p(\xi)\|_2
    \le E.
\end{equation}

\subsection{A gap-preserving reduction framework}
\label{subsec:gap-preserving-exponential-approximation}

We now develop a gap-preserving reduction framework for the two
transformed problems. Given a polynomial \(p\), we construct a family
\(q_h\) that is the exponential form of a signomial and satisfies the
Hessian perturbation bound~\autoref{eq:hessian-perturbation} for
sufficiently small \(h>0\). The construction applies a coordinatewise
deformation $\phi_h$ that approaches the identity as \(h\to0\), so that the
Hessian of \(q_h\) uniformly approximates that of \(p\) on the unit
box. This approximation preserves the curvature gap.

\paragraph{Curvature requirements.}
Let \(\phi_h:\mathbb{R}\to\mathbb{R}\) be twice continuously
differentiable and depend on \(h>0\). Apply it coordinatewise as
\(\Phi_h(y)=(\phi_h(y_1),\ldots,\phi_h(y_n))^\top\), and set
\(q_h(y)=p(\Phi_h(y))\). 
Thus \(q_h\) is obtained by evaluating \(p\)
at a coordinatewise deformation of \(y\).

To understand how this deformation changes curvature, write
\(J_h(y)=\operatorname{Diag}(\phi_h'(y_1),\ldots,\phi_h'(y_n))\).
The Hessian chain rule gives
\(\nabla^2q_h
=
J_h^\top\nabla^2p(\Phi_h)J_h+\operatorname{Diag}(
\phi_h''(y_i)\partial_i p(\Phi_h(y)))_{i=1}^n\).
Adding and subtracting \(\nabla^2p(\Phi_h(y))\) yields
\begin{align}
\label{eq:hessian-error-decomp}
\nabla^2q_h(y)-\nabla^2p(y)
&=
\underbrace{
J_h(y)^\top\nabla^2p(\Phi_h(y))J_h(y)
-\nabla^2p(\Phi_h(y))
}_{\text{Jacobian error}}
\nonumber\\
&\quad+
\underbrace{
\nabla^2p(\Phi_h(y))-\nabla^2p(y)
}_{\text{argument error}}
+
\underbrace{
\operatorname{Diag}\!\left(
\phi_h''(y_i)\,
\partial_i p(\Phi_h(y))
\right)_{i=1}^n
}_{\text{second-derivative error}} .
\end{align}

\autoref{eq:hessian-error-decomp} identifies three sources of
error. Uniformly on \([-1,1]\), the argument error reflects the fact
that the Hessian of \(p\) is evaluated at \(\Phi_h(y)\) rather than at
\(y\); it therefore becomes small when \(\phi_h(t)\to t\). The
Jacobian error is caused by the local rescaling \(J_h(y)\) introduced by
the coordinate map, and becomes small when \(\phi_h'(t)\to1\). The
second-derivative error is generated by the nonlinearity of
\(\phi_h\), and becomes small when \(\phi_h''(t)\to0\).

To quantify the three errors in~\autoref{eq:hessian-error-decomp}, we measure how closely
\(\phi_h\) follows the identity map on \([-1,1]\).
Let
\(E_0(h)=\sup_{|t|\le1}|\phi_h(t)-t|\), which measures how far the
transformed argument can move from the original one;
\(E_1(h)=\sup_{|t|\le1}|\phi_h'(t)-1|\), which measures how far the
local scaling deviates from the identity; and
\(E_2(h)=\sup_{|t|\le1}|\phi_h''(t)|\), which measures the
nonlinearity of the coordinate map.
The first quantity also determines where the transformed argument can
lie. If \(E_0(h)\le1\), then
\(\Phi_h([-1,1]^n)\subseteq[-2,2]^n\), since
\(|\phi_h(t)|\le |t|+|\phi_h(t)-t|\le2\).
It therefore suffices to control the derivatives of \(p\) on the fixed
enlarged box \([-2,2]^n\).

We next introduce the derivative bounds needed to control the three errors. Let \(G_p\) be a rational upper bound on the gradient of \(p\): $G_p\ge\sup_{\xi\in[-2,2]^n}\|\nabla p(\xi)\|_\infty$, and let \(H_p\) be a rational upper bound on its Hessian:
$H_p\ge \sup_{\xi\in[-2,2]^n}\|\nabla^2p(\xi)\|_2$.
Since \(p\) is a polynomial, \(\nabla^2p\) is Lipschitz on the compact box \([-2,2]^n\). Let \(L_p\) be a rational Lipschitz bound satisfying
    $\|\nabla^2p(\xi)-\nabla^2p(\zeta)\|_2
    \le
    L_p\|\xi-\zeta\|_\infty$
for $\xi,\zeta\in[-2,2]^n$.
Thus \(L_p\) controls how rapidly the Hessian can vary across the box. 

These quantities control the three errors directly. The argument error is
at most \(L_pE_0(h)\), because \(\Phi_h(y)\) lies within \(E_0(h)\) of
\(y\). For the Jacobian error, writing \(J_h=I_n+\Delta_h\) gives
\(\|\Delta_h\|_2\le E_1(h)\), so its norm is bounded by
\(H_p(2E_1(h)+E_1(h)^2)\). Finally, the second-derivative error is
diagonal, and its spectral norm is at most \(G_pE_2(h)\). Combining the
three bounds gives
\begin{equation}
\label{eq:curvature-transfer-bound}
    \sup_{y\in[-1,1]^n}
    \|\nabla^2q_h(y)-\nabla^2p(y)\|_2
    \le
    L_pE_0(h)
    +
    H_p\bigl(2E_1(h)+E_1(h)^2\bigr)
    +
    G_pE_2(h).
\end{equation}
Hence, if \(E_0(h)\), \(E_1(h)\), and \(E_2(h)\) are all \(O(h)\),
then the entire Hessian perturbation is \(O(h)\).

\paragraph{Canonical exponential deformation.}
The error \(E_0(h)\) measures how closely \(\phi_h\) approximates the
identity, so we require \(\phi_h(t)\to t\) as \(h\to0\). At the same
time, \(\phi_h\) should be built from exponentials so that
\(p\circ\Phi_h\) is the exponential form of a signomial. The simplest
starting point is \(e^{ht}\). For small h, Taylor's first-order expansion gives
\(e^{ht}=1+ht+O(h^2)\). 
This suggests the choice
\begin{equation}
\label{eq:canonical-exponential-map}
    \phi_h(t)=\frac{e^{ht}-1}{h}.
\end{equation}
We then have
$({e^{ht}-1})/{h}=t+O(h)$, so \(\phi_h(t)\to t\) as \(h\to0\).
It remains to verify the derivative requirements. Direct
differentiation gives
\(\phi_h'(t)=e^{ht}\) and
\(\phi_h''(t)=he^{ht}\), so
\(\phi_h'(t)\to1\) and \(\phi_h''(t)\to0\) as \(h\to0\).
More precisely, for \(|t|\leq1\) and \(0<h\leq1/2\), Taylor's theorem
yields
\(E_0(h)\leq h\), \(E_1(h)\leq2h\), and \(E_2(h)\leq2h\).
Thus, all three errors are \(O(h)\), and substituting these bounds into~\autoref{eq:curvature-transfer-bound} shows that the Hessian perturbation is also \(O(h)\).

Moreover, substituting
\(\phi_h(y_i)=(e^{hy_i}-1)/h\) into the finitely many monomials of
\(p\) and expanding their integer powers expresses \(q_h=p\circ\Phi_h\)
as a finite sum of terms \(c_\ell e^{h\ell^\top y}\). Thus, for
rational \(h>0\), \(q_h\) is the exponential form of a signomial with
rational coefficients and rational exponents.

The preceding construction shows that the Hessian
perturbation vanishes as \(h\to0\). To apply the gap-preservation
principle, we need a quantitative bound that determines how small
\(h\) must be for~\autoref{eq:hessian-perturbation} to hold. The
following lemma provides this bound.

\begin{lemma}[Gap-preserving curvature approximation]
\label{lem:exponential-curvature-approximation}
Let \(p:\mathbb{R}^n\to\mathbb{R}\) be a polynomial,
and let \(q_h=p\circ\Phi_h\). Then \(q_h\) is the exponential form of a
signomial and, for every rational \(0<h\leq1/2\),
\[
    \sup_{y\in[-1,1]^n}
    \left\|
        \nabla^2q_h(y)-\nabla^2p(y)
    \right\|_2
    \leq C_{\mathrm{err}}(p)h,
\]
where \(C_{\mathrm{err}}(p)=L_p+6H_p+2G_p\).
\end{lemma}

Thus, the Hessian perturbation condition~\autoref{eq:hessian-perturbation} holds with
\(E=C_{\mathrm{err}}(p)h\). Let
\(\gamma=\frac12\min\{\gamma_+,\gamma_-\}\) and
\(\overline C_{\mathrm{err}}(p)=\max\{1,C_{\mathrm{err}}(p)\}\), and
choose
\(h=\min\{\frac12,\gamma/\overline C_{\mathrm{err}}(p)\}\).
Since \(C_{\mathrm{err}}(p)\leq\overline C_{\mathrm{err}}(p)\), this
choice gives
\(E\leq\overline C_{\mathrm{err}}(p)h\leq\gamma
=\frac12\min\{\gamma_+,\gamma_-\}\).
Hence the two alternatives retain curvature margins of at least
\(\gamma_+/2\) and \(\gamma_-/2\), respectively. In our setting,
\(\overline C_{\mathrm{err}}(p)\) is polynomially bounded, while
\(\gamma_+\) and \(\gamma_-\) have inverse-polynomial magnitude.
Therefore, \(h\) can be chosen rational with inverse-polynomial
magnitude and polynomial encoding length. The resulting \(q_h\) can
be constructed in polynomial time and also has polynomial encoding
length. The details are given in Appendix~\ref{app:exponential-approximation-proof}.

\section{Constrained Hardness for Transformed Convexity}
\label{sec:compact-hardness}

We now establish strong NP-hardness for the constrained versions of the two transformed convexity problems. Both reductions apply the gap-preserving exponential approximation of
Section~\ref{subsec:gap-preserving-exponential-approximation} on a compact domain, with the approximation parameter chosen so that the Hessian perturbation remains smaller than the curvature margins of \CBGC{}. We begin with log-coordinate convexity.

\begin{theorem}
\label{thm:compact-log-coordinate}
The constrained \LCC{} problem is strongly NP-hard.
\end{theorem}

\begin{proof}
Let \(p\) be an instance of \CBGC{} on \(Y=[-1,1]^n\), with curvature
margins \(\gamma_+\) and \(\gamma_-\).
By Lemma~\ref{lem:exponential-curvature-approximation}, we can choose a positive
rational \(h\), with polynomial encoding length, and construct a
signomial \(q_h(y)=f_h(e^y)\) such that
    $\sup_{y\in Y}
    \|\nabla^2q_h(y)-\nabla^2p(y)\|_2
    \leq E$,
    $E\leq\frac12\min\{\gamma_+,\gamma_-\}$.
By the gap-preserving argument of
Section~\ref{subsec:gap-convexity}, \(q_h\) is convex on \(Y\) in the positive case and nonconvex on \(Y\) in the negative case. Since
\(q_h(y)=f_h(e^y)\), the signomial \(f_h\) is log-coordinate convex on
\(Y\) exactly in the positive case.

The construction of \(f_h\) is computable in polynomial time and preserves
the required coefficient and exponent encoding bounds. Thus, an algorithm for the constrained version of
\textsc{Log-Coordinate-Convexity} would solve \CBGC{}. Since \CBGC{} is
strongly NP-hard and the reduction preserves the required encoding
bounds, the result follows.
\end{proof}

We next consider log-log convexity. 
For any positive twice
continuously differentiable function \(g\),
$\nabla^2\log g(y)=
({g(y)\nabla^2g(y)-\nabla g(y)\nabla g(y)^\top})/{g(y)^2}$.
Since \(g(y)^2>0\), the function \(\log g\) is convex exactly when the numerator is positive semidefinite throughout the domain. We therefore define the log-curvature matrix by
\begin{equation}
\label{eq:log-curvature-matrix}
M_g(y)
=
g(y)\nabla^2g(y)
-
\nabla g(y)\nabla g(y)^\top.
\end{equation}
Our reduction adds
a positive constant to the exponential approximation and scales its
remaining terms so that \(g\) stays positive and, in the positive
alternative, \(g\nabla^2g\) dominates 
\(\nabla g\nabla g^\top\).

\begin{theorem}
\label{thm:compact-log-log}
The constrained \LLC{} problem is strongly NP-hard.
\end{theorem}

\begin{proof}
Let \(p\) be an instance of \CBGC{} on \(Y=[-1,1]^n\). As in the proof
of Theorem~\ref{thm:compact-log-coordinate}, choose a positive rational
\(h\) and construct \(q_h(y)=f_h(e^y)\). The gap-preserving
approximation retains inverse-polynomial margins
\(\gamma_+,\gamma_->0\), so
$\kappa_Y(q_h)\geq\gamma_+$ in the positive case,
$\kappa_Y(q_h)\leq-\gamma_-$ in the negative case.

Choose rational bounds \(B_0,B_1\geq1\) such that
\(|q_h(y)|\leq B_0\) and
\(\|\nabla q_h(y)\|_2\leq B_1\) for every \(y\in Y\), and take a
positive rational
\(\varepsilon\leq
\min\{1/(2B_0),\gamma_+/(4B_1^2)\}\).
Define \(g_h(y)=1+\varepsilon q_h(y)\). Then \(g_h\geq1/2\) on
\(Y\), and
$M_{g_h}=\varepsilon g_h\nabla^2q_h
-\varepsilon^2\nabla q_h\nabla q_h^\top$.

In the positive case, Weyl's inequality gives, for every \(y\in Y\),
\[
\begin{aligned}
    \lambda_{\min}\!\left(M_{g_h}(y)\right)
    \geq
    \varepsilon g_h(y)
    \lambda_{\min}\!\left(\nabla^2q_h(y)\right)
    -
    \varepsilon^2\|\nabla q_h(y)\|_2^2 \geq
    \varepsilon
    \left(
        \frac{\gamma_+}{2}-\varepsilon B_1^2
    \right)
    \geq
    \frac{\varepsilon\gamma_+}{4}.
\end{aligned}
\]
Hence \(M_{g_h}(y)\succeq0\) throughout \(Y\), and therefore
\(\log g_h\) is convex on \(Y\).

In the negative case, there exists \(\hat y\in Y\) such that
\(\lambda_{\min}(\nabla^2q_h(\hat y))\leq-\gamma_-\). Since the
rank-one term in \(M_{g_h}\) is negative semidefinite,
\[
    \lambda_{\min}\!\left(M_{g_h}(\hat y)\right)
    \leq
    \varepsilon g_h(\hat y)
    \lambda_{\min}\!\left(\nabla^2q_h(\hat y)\right)
    \leq
    -\varepsilon g_h(\hat y)\gamma_-
    <0.
\]
Thus, \(\log g_h\) is not convex on \(Y\).

The exponential-approximation bounds allow \(B_0\), \(B_1\), and
\(\varepsilon\) to be chosen with polynomial encoding length. Since
\(q_h\) is the exponential form of a signomial, so is
\(g_h=1+\varepsilon q_h\). The corresponding signomial can therefore
be constructed in polynomial time while preserving the required
encoding bounds. This gives a reduction from \CBGC{}, and strong
NP-hardness follows.
\end{proof}

\section{Unconstrained hardness for transformed convexity}
\label{sec:unconstrained-hardness}

We now establish strong NP-hardness for the unconstrained versions of the two transformed convexity problems. Compact-domain hardness need not extend to \(\mathbb{R}^n\): an interior negative-curvature witness remains valid globally, but convexity on the unit box gives no control outside it. To bridge this gap, both reductions augment the same hard exponential form \(q\) with an auxiliary construction that extends the curvature separation beyond the unit box while preserving the interior witness.

Let \(p\) be an instance of \GQGC{}, choose a rational \(h\) according to the gap-preservation condition of Section~\ref{subsec:gap-preserving-exponential-approximation}, and write \(q=q_h=p\circ\Phi_h\). Since \(p\) has degree four and \(h\le\frac12\),
$D=\max_{\alpha\in\operatorname{supp}(q)}\|\alpha\|_1\le2$.
Let \(\gamma=\frac12\min\{\gamma_+,\gamma_-\}\). By Lemma~\ref{lem:exponential-curvature-approximation} and the interior-witness property proved in Appendix~\ref{app:quartic-gap-proof}, it is strongly NP-hard to distinguish
$\kappa_{[-1,1]^n}(q)\ge\gamma$
from
$\kappa_{[-\frac12,\frac12]^n}(q)\le-\gamma$.

To extend this curvature separation to all of \(\mathbb{R}^n\), we
augment \(q\) with an auxiliary exponential form:
\begin{equation}
\label{eq:globalized-exponential-form}
Q_L(y)=q(y)+\delta_L\Psi_L(y).
\end{equation}
Here, \(\Psi_L\) is designed to provide sufficiently large positive
curvature outside the unit box to dominate any negative curvature of
\(q\). The coefficient \(\delta_L>0\) will be chosen small enough that
\(Q_L\) retains the interior negative curvature of \(q\), while \(L\)
controls how rapidly the curvature supplied by \(\Psi_L\) grows outside
the unit box.

A natural first choice for \(\Psi_L\) is
\(\sum_i(e^{Ly_i}+e^{-Ly_i})\), since, whenever \(|y_i|\) is large,
one of \(e^{Ly_i}\) and \(e^{-Ly_i}\) grows as \(e^{L|y_i|}\).
However, its Hessian is diagonal, so the rapidly growing curvature
generated by \(y_i\) is confined to the \(e_i\)-direction and may not
dominate negative curvature of \(q\) in other directions. To distribute
this growth across every coordinate direction, we couple
\(e^{Ly_i}+e^{-Ly_i}\) with \(e^{y_j}+e^{-y_j}\) for every \(i,j\)
and define
\begin{equation}
\label{eq:global-auxiliary}
\Psi_L(y)
=
\sum_{i,j=1}^n
\bigl(e^{Ly_i}+e^{-Ly_i}\bigr)
\bigl(e^{y_j}+e^{-y_j}\bigr) 
=
\sum_{i,j=1}^n
\sum_{s,t\in{-1,1}}
e^{sLy_i+t y_j},
\qquad L>1.
\end{equation}
In the factorized representation, \(e^{Ly_i}+e^{-Ly_i}\) provides
rapid growth when \(|y_i|\) is large, while
\(e^{y_j}+e^{-y_j}\) transfers this growth into curvature in the
\(e_j\)-direction. Summing over all \(i,j\) ensures that, regardless
of which coordinate of \(y\) is large, \(\Psi_L\) provides
growing curvature in every direction.

The properties used in both reductions are summarized below.

\begin{lemma}[Globalizing auxiliary]
\label{lem:auxiliary-exponential}
The function \(\Psi_L\) is positive, globally convex, and globally
log-convex. Writing \(R=\|y\|_\infty\), and letting
\(M_{\Psi_L}\) denote the log-curvature matrix of \(\Psi_L\), the
smallest eigenvalues of its Hessian and log-curvature matrix satisfy
$\lambda_{\min}(\nabla^2\Psi_L(y))
\geq2e^{(L-1)R}$,
$\lambda_{\min}(M_{\Psi_L}(y))\geq4e^{2LR}$.
Moreover, there
exist polynomials \(P_1(n,L)\) and \(P_2(n,L)\) such that
\(\Psi_L(y)\), \(\|\nabla\Psi_L(y)\|_2\), and
\(\|\nabla^2\Psi_L(y)\|_2\) are bounded above by
\(P_1(n,L)e^{(L+1)R}\), while
$\|M_{\Psi_L}(y)\|_2 \leq P_2(n,L)e^{2(L+1)R}$.
\end{lemma}

The lower bounds allow the auxiliary term to dominate the curvature of
\(q\) outside the unit box, while the upper bounds allow its coefficient
to be chosen small enough to preserve negative curvature inside the
inner box. The proof and explicit expressions for \(P_1\) and \(P_2\)
are given in Appendix~\ref{app:auxiliary-exponential-proof}.

We first apply the globalization construction to log-coordinate convexity.
We give the main argument below and defer the complete quantitative proof to
Appendix~\ref{app:global-log-coordinate-proof}.

\begin{theorem}
\label{thm:global-log-coordinate}
The unconstrained \LCC{} problem is strongly NP-hard.
\end{theorem}
\begin{proof}[Proof sketch]
Consider \(Q_L=q+\delta_L\Psi_L\). In the negative alternative,
$\kappa_{[-\frac12,\frac12]^n}(q)\leq-\gamma$.
On \([-\frac12,\frac12]^n\), Lemma~\ref{lem:auxiliary-exponential} gives
$\|\nabla^2\Psi_L\|_2
\leq P_1(n,L)e^{(L+1)/2}$.
We therefore choose \(\delta_L>0\) such that
$\delta_LP_1(n,L)e^{(L+1)/2}\leq\frac{\gamma}{2}$.
By the curvature perturbation bound,
$\kappa_{[-\frac12,\frac12]^n}(Q_L)
\leq-\frac{\gamma}{2}$,
so \(Q_L\) remains nonconvex.

In the positive alternative,
$\kappa_{[-1,1]^n}(q)\geq\gamma$.
Since \(\Psi_L\) is convex, \(Q_L\) is convex on the unit box. For
\(R=\|y\|_\infty\geq1\), the exponent bound \(D\leq2\) gives
$\|\nabla^2q(y)\|_2\leq C_2e^{DR}$
for a polynomially bounded constant \(C_2\). Lemma~\ref{lem:auxiliary-exponential} and Weyl's inequality then give
\[
\lambda_{\min}\!\left(\nabla^2Q_L(y)\right)
\geq
2\delta_Le^{(L-1)R}-C_2e^{DR}.
\]
The ratio of the auxiliary curvature term to the adverse curvature term is
\(2\delta_L e^{(L-1)R}/(C_2e^{DR})
=(2\delta_L/C_2)e^{(L-1-D)R}\).
When \(L-1>D\), this ratio increases with \(R\). Hence it suffices to make the preceding lower bound nonnegative at \(R=1\).
The quantitative bounds show that an integer \(L=O(\log n)\) and a positive rational \(\delta_L\) of inverse-polynomial magnitude can simultaneously preserve the negative alternative and ensure
$\lambda_{\min}\!\left(\nabla^2Q_L(y)\right)\geq0$
whenever $\|y\|_\infty\geq1$.

Thus, \(Q_L\) is convex on \(\mathbb{R}^n\) in the positive alternative and remains nonconvex in the negative alternative. Since \(Q_L\) is the exponential form of a signomial and \(L\) and \(\delta_L\) have polynomial encoding length, this gives a polynomial-time reduction from \GQGC{}. Strong NP-hardness follows.
\end{proof}

We next adapt the globalization construction to log-log convexity,
where curvature is governed by the log-curvature matrix $M$ rather than the Hessian.
We give the main argument below and defer the complete quantitative proof to Appendix~\ref{app:global-loglog-proof}.

\begin{theorem}
\label{thm:global-log-log-convexity-hardness}
The unconstrained \LLC{} problem is strongly NP-hard.
\end{theorem}
\begin{proof}[Proof sketch]
Starting from the same exponential polynomial \(q\), choose
\(\varepsilon>0\) as in the compact log-log reduction and define
$g_0=1+\varepsilon q$.
The choice of \(\varepsilon\) ensures that \(g_0\ge\frac12\) on the unit box in both alternatives. In the positive alternative, \(g_0\) is log-convex there; in the negative alternative, it retains a strict
\(M_{g_0}\)-curvature witness in \([-\frac12,\frac12]^n\). We augment
\(g_0\) by setting
$G_L=g_0+\delta_L\Psi_L$.

Unlike the Hessian, \(M_g=g\nabla^2g-\nabla g\nabla g^\top\) is
nonlinear under addition. For twice continuously differentiable
functions \(u\) and \(v\), define
$\mathcal{C}(u,v)
=
u\nabla^2v+v\nabla^2u
-\nabla u\nabla v^\top
-\nabla v\nabla u^\top$.
A direct expansion gives
$M_{G_L}
=
M_{g_0}
+\delta_L\mathcal{C}(g_0,\Psi_L)
+\delta_L^2M_{\Psi_L}$.
For a fixed \(L\), we first choose \(\delta_L>0\) sufficiently small
that the last two terms preserve the negative eigenvalue of
\(M_{g_0}\) at the interior witness. Hence, \(G_L\) remains non-log-convex
in the negative alternative.
In the positive alternative, both \(g_0\) and \(\Psi_L\) are positive
and log-convex on the unit box. Since positive sums of log-convex
functions are log-convex, \(G_L\) is log-convex there.

Now let \(R=\|y\|_\infty\geq1\). The growth bounds for \(q\), together
with Lemma~\ref{lem:auxiliary-exponential}, give polynomially bounded
constants \(A_0\) and \(A_1\) such that
$\|M_{g_0}(y)\|_2\leq A_0e^{2DR}$,
$\|\mathcal{C}(g_0,\Psi_L)(y)\|_2
\leq A_1e^{(D+L+1)R}$.
Using the lower bound on \(M_{\Psi_L}\), we obtain
\[
\lambda_{\min}\!\left(M_{G_L}(y)\right)
\geq
4\delta_L^2e^{2LR}
-A_0e^{2DR}
-\delta_LA_1e^{(D+L+1)R}.
\]
When \(L>D+1\), the positive term grows faster than both potentially
negative terms. It therefore suffices to make this lower bound
nonnegative at \(R=1\).
The quantitative estimates show that an integer \(L=O(\log n)\) and a positive rational \(\delta_L\) of inverse-polynomial magnitude can be chosen to satisfy this condition while preserving the interior negative-curvature witness.

It remains to establish global positivity in both alternatives. For some polynomially bounded
constant \(C_0\), we have
$G_L(y)
\geq
1-\varepsilon C_0e^{DR}
+2\delta_Le^{LR}$.
Because \(L>D\), the auxiliary term grows faster than the potentially
negative contribution from \(q\), so positivity outside the unit box can likewise be enforced at \(R=1\). 
The same \(L\) and \(\delta_L\) chosen above also satisfy this condition.
Together with \(g_0\ge\frac12\) and \(\Psi_L>0\) on the unit box, this proves that \(G_L>0\) on \(\mathbb{R}^n\) in both alternatives.

Consequently, \(G_L\) is globally log-convex in the positive alternative and remains positive but non-log-convex in the negative alternative. Since \(G_L\) is the exponential form of a signomial and \(L,\delta_L\) have polynomial encoding length, this gives a polynomial-time reduction from \GQGC{}, proving strong NP-hardness.
\end{proof}

\section{Conclusion}

We established that recognizing convexity, log-coordinate convexity, and log-log convexity of signomials is strongly NP-hard, both on compact domains and globally. Our reductions transfer inverse-polynomial curvature gaps from low-degree polynomials through exponential approximations and use auxiliary exponential terms to extend this separation beyond a compact region without destroying an interior nonconvexity witness. Thus, although these properties can enable tractable optimization once certified, no polynomial-time method can recognize them for arbitrary signomials unless \(\mathrm{P}=\mathrm{NP}\). These worst-case results leave room for efficient methods on structured subclasses and motivate sound, efficiently verifiable certificates and systematically improvable approximation hierarchies for recognizing convex structure in signomial models.

\bibliographystyle{plain}
\bibliography{ref}

@article{murray2021signomial,
  title={Signomial and polynomial optimization via relative entropy and partial dualization},
  author={Murray, Riley and Chandrasekaran, Venkat and Wierman, Adam},
  journal={Mathematical Programming Computation},
  volume={13},
  number={2},
  pages={257--295},
  year={2021},
  publisher={Springer}
}

@article{wang2020positivstellensatz,
  title={A {Positivstellensatz} for conditional {SAGE} signomials},
  author={Wang, Allen Houze and Jaini, Priyank and Yu, Yaoliang and Poupart, Pascal},
  journal={arXiv preprint arXiv:2003.03731},
  year={2020}
}

@article{dressler2022algebraic,
  title={Algebraic perspectives on signomial optimization},
  author={Dressler, Mareike and Murray, Riley},
  journal={SIAM Journal on Applied Algebra and Geometry},
  volume={6},
  number={4},
  pages={650--684},
  year={2022},
  publisher={SIAM}
}

@article{murray2021newton,
  title={Newton polytopes and relative entropy optimization},
  author={Murray, Riley and Chandrasekaran, Venkat and Wierman, Adam},
  journal={Foundations of Computational Mathematics},
  volume={21},
  number={6},
  pages={1703--1737},
  year={2021},
  publisher={Springer}
}

@article{ahmadi2013np,
  title={{NP}-hardness of deciding convexity of quartic polynomials and related problems},
  author={Ahmadi, Amir Ali and Olshevsky, Alex and Parrilo, Pablo A and Tsitsiklis, John N},
  journal={Mathematical Programming},
  volume={137},
  number={1},
  pages={453--476},
  year={2013},
  publisher={Springer}
}

@article{ahmadi2020complexity,
  title={On the complexity of detecting convexity over a box},
  author={Ahmadi, Amir Ali and Hall, Georgina},
  journal={Mathematical Programming},
  volume={182},
  number={1},
  pages={429--443},
  year={2020},
  publisher={Springer}
}

@article{chandrasekaran2016relative,
  title={Relative entropy relaxations for signomial optimization},
  author={Chandrasekaran, Venkat and Shah, Parikshit},
  journal={SIAM Journal on Optimization},
  volume={26},
  number={2},
  pages={1147--1173},
  year={2016},
  publisher={SIAM}
}

@article{ling2010biquadratic,
  title={Biquadratic optimization over unit spheres and semidefinite programming relaxations},
  author={Ling, Chen and Nie, Jiawang and Qi, Liqun and Ye, Yinyu},
  journal={SIAM Journal on Optimization},
  volume={20},
  number={3},
  pages={1286--1310},
  year={2010},
  publisher={SIAM}
}

@article{motzkin1965maxima,
  title={Maxima for graphs and a new proof of a theorem of {Tur{\'a}n}},
  author={Motzkin, Theodore S and Straus, Ernst G},
  journal={Canadian Journal of Mathematics},
  volume={17},
  pages={533--540},
  year={1965},
  publisher={Cambridge University Press}
}

@article{boyd2007tutorial,
  title={A tutorial on geometric programming},
  author={Boyd, Stephen and Kim, Seung-Jean and Vandenberghe, Lieven and Hassibi, Arash},
  journal={Optimization and Engineering},
  volume={8},
  number={1},
  pages={67--127},
  year={2007},
  publisher={Springer}
}

@article{karcher2023data,
  title={Data fitting with signomial programming compatible difference of convex functions},
  author={Karcher, Cody J},
  journal={Optimization and Engineering},
  volume={24},
  number={2},
  pages={973--987},
  year={2023},
  publisher={Springer}
}

@article{agrawal2019disciplined,
  title={Disciplined geometric programming},
  author={Agrawal, Akshay and Diamond, Steven and Boyd, Stephen},
  journal={Optimization Letters},
  volume={13},
  number={5},
  pages={961--976},
  year={2019},
  publisher={Springer}
}

@article{lee2014sparse,
  title={The sparse signomial classification and regression model},
  author={Lee, Kyungsik and Kim, Norman and Jeong, Myong K},
  journal={Annals of Operations Research},
  volume={216},
  number={1},
  pages={257--286},
  year={2014},
  publisher={Springer}
}

@article{lumadjeng2026ecsel,
  title={{ECSEL}: Explainable Classification via Signomial Equation Learning},
  author={Lumadjeng, Adia and Birbil, Ilker and Acar, Erman},
  journal={arXiv preprint arXiv:2601.21789},
  year={2026}
}

@book{bhatia2013matrix,
  title={Matrix analysis},
  author={Bhatia, Rajendra},
  year={2013},
  publisher={Springer Science \& Business Media}
}

@article{bjork2002global,
  title={Global optimization of heat exchanger network synthesis problems with and without the isothermal mixing assumption},
  author={Bj{\"o}rk, Kaj-Mikael and Westerlund, Tapio},
  journal={Computers \& Chemical Engineering},
  volume={26},
  number={11},
  pages={1581--1593},
  year={2002},
  publisher={Elsevier}
}

@article{chiang2007power,
  title={Power control by geometric programming},
  author={Chiang, Mung and Tan, Chee Wei and Palomar, Daniel P and O'Neill, Daniel and Julian, David},
  journal={IEEE Transactions on Wireless Communications},
  volume={6},
  number={7},
  pages={2640--2651},
  year={2007},
  publisher={IEEE}
}

@article{chiu2007optimal,
  title={Optimal production lot sizing with rework, scrap rate, and service level constraint},
  author={Chiu, Singa Wang and Ting, Chia-Kuan and Chiu, Yuan-Shyi Peter},
  journal={Mathematical and Computer Modelling},
  volume={46},
  number={3-4},
  pages={535--549},
  year={2007},
  publisher={Elsevier}
}

@article{fevotte2011algorithms,
  title={Algorithms for nonnegative matrix factorization with the $\beta$-divergence},
  author={F{\'e}votte, C{\'e}dric and Idier, J{\'e}r{\^o}me},
  journal={Neural Computation},
  volume={23},
  number={9},
  pages={2421--2456},
  year={2011},
  publisher={MIT Press}
}

@article{kirschen2018application,
  title={Application of signomial programming to aircraft design},
  author={Kirschen, Philippe G and York, Martin A and Ozturk, Berk and Hoburg, Warren W},
  journal={Journal of Aircraft},
  volume={55},
  number={3},
  pages={965--987},
  year={2018},
  publisher={American Institute of Aeronautics and Astronautics}
}

@article{maranas1997global,
  title={Global optimization in generalized geometric programming},
  author={Maranas, Costas D and Floudas, Christodoulos A},
  journal={Computers \& Chemical Engineering},
  volume={21},
  number={4},
  pages={351--369},
  year={1997},
  publisher={Elsevier}
}

@article{marin2007optimization,
  title={Optimization of biotechnological systems through geometric programming},
  author={Marin-Sanguino, Alberto and Voit, Eberhard O and Gonzalez-Alcon, Carlos and Torres, Nestor V},
  journal={Theoretical Biology and Medical Modelling},
  volume={4},
  number={1},
  pages={38},
  year={2007},
  publisher={Springer}
}

@article{san2022inventory,
  title={An inventory model for multiple items assuming time-varying demands and limited storage},
  author={San-Jos{\'e}, Luis A and Gonz{\'a}lez-De-la-Rosa, Manuel and Sicilia, Joaqu{\'\i}n and Febles-Acosta, Jaime},
  journal={Optimization Letters},
  volume={16},
  number={6},
  pages={1935--1961},
  year={2022},
  publisher={Springer}
}

@article{zimmert2021tsallis,
  title={Tsallis-{INF}: An optimal algorithm for stochastic and adversarial bandits},
  author={Zimmert, Julian and Seldin, Yevgeny},
  journal={Journal of Machine Learning Research},
  volume={22},
  number={28},
  pages={1--49},
  year={2021}
}

@article{murty1987some,
  title={Some {NP}-complete problems in quadratic and nonlinear programming},
  author={Murty, Katta G and Kabadi, Santosh N},
  journal={Mathematical Programming},
  volume={39},
  number={2},
  pages={117--129},
  year={1987},
  publisher={Springer Berlin Heidelberg Berlin/Heidelberg}
}

@article{udrescu2020ai,
  title={{AI} {F}eynman: A physics-inspired method for symbolic regression},
  author={Udrescu, Silviu-Marian and Tegmark, Max},
  journal={Science Advances},
  volume={6},
  number={16},
  pages={eaay2631},
  year={2020},
  publisher={American Association for the Advancement of Science}
}

@article{hwang2015multi,
  title={Multi-class classification using a signomial function},
  author={Hwang, Kyoungmi and Lee, Kyungsik and Lee, Chungmok and Park, Sungsoo},
  journal={Journal of the Operational Research Society},
  volume={66},
  number={3},
  pages={434--449},
  year={2015},
  publisher={Taylor \& Francis}
}

@incollection{grant2006disciplined,
  title={Disciplined convex programming},
  author={Grant, Michael and Boyd, Stephen and Ye, Yinyu},
  booktitle={Global optimization: From theory to implementation},
  pages={155--210},
  year={2006},
  publisher={Springer}
}

\appendix
\section{Proofs for Section~\ref{sec:convexity-problems-and-gap-reductions}}
\label{app:gap-proofs}

\subsection{Global Quartic Gap-Convexity (\GQGC{}) Hardness}
\label{app:quartic-gap-proof}

\begin{proof}
We combine the biquadratic hardness construction of
\cite{ling2010biquadratic} with the quartic convexity reduction of
\cite[Theorem~2.3]{ahmadi2013np}.

Let $(G,K)$ be an instance of \textsc{Independent Set}, where $G$ has
$m$ vertices and $2 \le K \le m$, and set $\tau=K-1$.
The construction below has ambient dimension $n=2m$.
Let $\bar G$ denote the complement of $G$. Following
\cite{ling2010biquadratic}, define
\[
    b_{\bar G}(x;y)
    =
    -2\sum_{(i,j)\in E(\bar G)} x_i x_j y_i y_j .
\]
Let $\operatorname{ind}(G)$ denote the independence number of $G$.
By the Motzkin--Straus identity~\citep{motzkin1965maxima},
\(\min_{\|x\|_2=\|y\|_2=1}b_{\bar G}(x;y)
=-1+1/\operatorname{ind}(G)\).

Now set
\[
    b(x;y) = \tau b_{\bar G}(x;y)
    +
    (\tau-1)\|x\|_2^2\|y\|_2^2 .
\]
On the product of unit spheres,
\(\min b=\tau/\operatorname{ind}(G)-1\).
Hence, if \(\operatorname{ind}(G)<K\), then
\(\operatorname{ind}(G)\le\tau\), so \(b\ge0\) on the product of unit
spheres and therefore globally, by biquadratic homogeneity.
If \(\operatorname{ind}(G)\ge K=\tau+1\), there instead exist unit
vectors \(\bar x,\bar y\) such that
\(b(\bar x;\bar y)\le-1/K\).

We next apply the quartic construction of
\cite[Theorem~2.3]{ahmadi2013np}. Let \(C_b(x,y)\) be the matrix of
mixed second derivatives of \(b\), and let \(\Lambda_b\) be the largest
absolute coefficient appearing in its entries. Define
\[
\begin{aligned}
    P_b(x,y)
    =
    b(x;y)
    +
    \frac{m^2\Lambda_b}{2}
    \Bigg(
        \sum_{i=1}^m x_i^4
        +
        \sum_{i=1}^m y_i^4
        +
        \sum_{i<j}x_i^2x_j^2
        +
        \sum_{i<j}y_i^2y_j^2
    \Bigg).
\end{aligned}
\]
Their construction shows that \(b\) is globally nonnegative if and only
if \(P_b\) is globally convex.

Consider first the case \(\operatorname{ind}(G)\ge K\).
Let \(w_0=(\bar x,0)\) and \(d=(0,\bar y)\), so that
\(\|d\|_2=1\). The quartic construction gives
\[
    d^\top\nabla^2P_b(w_0)d
    =
    2b(\bar x;\bar y)
    \le
    -\frac{2}{K}.
\]
Since \(P_b\) is homogeneous of degree four, its Hessian is homogeneous
of degree two. Thus, with
\(\widehat w=w_0/2\in[-1/2,1/2]^n\),
\[
    d^\top\nabla^2P_b(\widehat w)d
    \le
    -\frac{1}{2K}.
\]

To create a uniform positive margin in the other alternative, define
\[
    p(w)
    =
    P_b(w)
    +
    \frac{1}{8K}\|w\|_2^2 .
\]
If \(\operatorname{ind}(G)<K\), then \(P_b\) is globally convex, and
therefore
\[
    \nabla^2p(w)
    \succeq
    \frac{1}{4K}I_n
    \qquad
    \text{for every }w\in\mathbb{R}^n.
\]
If \(\operatorname{ind}(G)\ge K\), then at the interior witness
\(\widehat w\),
\[
    d^\top\nabla^2p(\widehat w)d
    \le
    -\frac{1}{2K}
    +
    \frac{1}{4K}
    =
    -\frac{1}{4K}.
\]
By Rayleigh--Ritz,
\(\lambda_{\min}(\nabla^2p(\widehat w))\le-1/(4K)\).

Thus the two alternatives have the common margin
\(\gamma_+=\gamma_-=1/(4K)\), which has inverse-polynomial magnitude
because \(K\le m\). 
Moreover, \(p\) has degree at most four, contains only polynomially many
variables and monomials, and has rational coefficients whose numerators
and denominators are polynomially bounded in the input size. Hence \(p\)
has polynomial encoding length and can be constructed in polynomial
time. It follows that \GQGC{} is strongly NP-hard.
\end{proof}

\subsection{Proof of Lemma~\ref{lem:exponential-curvature-approximation}}
\label{app:exponential-approximation-proof}

\begin{proof}
Write
\[
    p(\xi)=\sum_{|\alpha|\le d}p_\alpha \xi^\alpha,
    \qquad
    S_p=\sum_{|\alpha|\le d}|p_\alpha|.
\]
We first verify the exponential representation. For each multi-index
\(\alpha\),
\[
\begin{aligned}
    \Phi_h(y)^\alpha
    =
    h^{-|\alpha|}
    \prod_{i=1}^n(e^{hy_i}-1)^{\alpha_i} =
    h^{-|\alpha|}
    \sum_{0\le\ell\le\alpha}
    \binom{\alpha}{\ell}
    (-1)^{|\alpha|-|\ell|}
    e^{h\ell^\top y}.
\end{aligned}
\]
Since \(|\alpha|\le d\), summing over the monomials of \(p\) gives
\(q_h(y)=\sum_{|\ell|\le d}c_\ell e^{h\ell^\top y}\) for rational
coefficients \(c_\ell\).

We next bound the Hessian perturbation. On \([-2,2]^n\), termwise
differentiation gives
\(\|\nabla p(\xi)\|_\infty\le G_p\) and
\(\|\nabla^2p(\xi)\|_2\le H_p\), where
\(G_p=d2^dS_p\) and \(H_p=nd^22^dS_p\).
Moreover, every third partial derivative is bounded in absolute value by
\(d^32^dS_p\). Applying the fundamental theorem of calculus along the
line segment joining \(\xi,\zeta\in[-2,2]^n\), together with
\(\|A\|_2\le\|A\|_F\le n\max_{i,j}|A_{ij}|\), gives
\[
    \|\nabla^2p(\xi)-\nabla^2p(\zeta)\|_2
    \le
    L_p\|\xi-\zeta\|_\infty,
    \qquad
    L_p=n^2d^32^dS_p.
\]

For \(|t|\le1\) and \(0<h\le\frac12\),
\(|\phi_h(t)-t|\le h\),
\(|\phi_h'(t)-1|\le2h\), and
\(|\phi_h''(t)|\le2h\).
Hence, for \(y\in[-1,1]^n\),
\(\|\Phi_h(y)-y\|_\infty\le h\),
\(\Phi_h(y)\in[-2,2]^n\), and, with
\(J_h(y)=\operatorname{Diag}(e^{hy_1},\ldots,e^{hy_n})\),
\(\|J_h(y)-I_n\|_2\le2h\) and
\(\|J_h(y)\|_2\le2\).

Using the decomposition in
\autoref{eq:hessian-error-decomp}, the Jacobian error satisfies
\[
\begin{aligned}
&
\left\|
J_h^\top\nabla^2p(\Phi_h)J_h
-\nabla^2p(\Phi_h)
\right\|_2 \\
&\qquad\le
\|J_h-I_n\|_2
\|\nabla^2p(\Phi_h)\|_2
\|J_h\|_2
+
\|\nabla^2p(\Phi_h)\|_2
\|J_h-I_n\|_2
\le 6H_ph.
\end{aligned}
\]
The argument error is at most
\(L_p\|\Phi_h(y)-y\|_\infty\le L_ph\).
Since the remaining term in
\autoref{eq:hessian-error-decomp} is diagonal, its spectral norm is at
most \(2G_ph\). Therefore
\[
    \|\nabla^2q_h(y)-\nabla^2p(y)\|_2
    \le
    (6H_p+L_p+2G_p)h
\]
uniformly on \([-1,1]^n\). Thus we may take
\(C_{\mathrm{err}}(p)=6H_p+L_p+2G_p\), which is rational,
computable in polynomial time from \(p\), and polynomially bounded in
the input size.

It remains to verify the encoding claim. Since \(d\) is fixed, the
expansion above contains only \(O(n^d)\) exponent vectors. Each
coefficient \(c_\ell\) is a finite sum of terms involving
\(p_\alpha\), binomial coefficients, and powers \(h^{-|\alpha|}\).
Hence, if \(h\) has polynomially bounded rational data, then so do the
coefficients \(c_\ell\) and exponents \(h\ell\). Thus \(q_h\) satisfies
the required encoding bounds.
\end{proof}

\section{Proofs for Section~\ref{sec:unconstrained-hardness}}
\label{app:global-proofs}

\subsection{Proof of Lemma~\ref{lem:auxiliary-exponential}}
\label{app:auxiliary-exponential-proof}

\begin{proof}
Positivity is immediate. Each summand of \(\Psi_L\) is the exponential
of an affine function and is therefore convex, so \(\Psi_L\) is globally
convex. Moreover, \(\log\Psi_L\) is a log-sum-exp of affine functions,
hence \(\Psi_L\) is globally log-convex.

Fix \(y\in\mathbb{R}^n\) and write \(R=\|y\|_\infty\). Choose an index
\(i\) and a sign \(s\in\{-1,1\}\) such that \(s y_i=R\).
For arbitrary \(z\in\mathbb{R}^n\) and any \(j\), pairing the two
terms corresponding to \(t=\pm1\) gives
\[
\begin{aligned}
&
e^{sLy_i+y_j}(sLz_i+z_j)^2
+
e^{sLy_i-y_j}(sLz_i-z_j)^2                                      \\
&\qquad\ge
e^{LR-|y_j|}
\left[
    (sLz_i+z_j)^2+(sLz_i-z_j)^2
\right]                                                          \\
&\qquad=
2e^{LR-|y_j|}
\left(L^2z_i^2+z_j^2\right)                                     \\
&\qquad\ge
2e^{(L-1)R}
\left(L^2z_i^2+z_j^2\right),
\end{aligned}
\]
where the last inequality uses \(|y_j|\le R\).
Summing over \(j\) and discarding the nonnegative terms involving
\(L^2z_i^2\) yields
\[
    z^\top\nabla^2\Psi_L(y)z
    \ge
    2e^{(L-1)R}\|z\|_2^2.
\]
Since it holds for every \(z\), we obtain
\(\nabla^2\Psi_L(y)\succeq2e^{(L-1)R}I_n\), thus
$\lambda_{\min}\left(\nabla^2\Psi_L(y)\right)\geq 2e^{(L-1)R}$.

We next record the required upper bounds. Every exponent vector appearing
in \(\Psi_L\) has \(\ell_1\)-norm at most \(L+1\), and there are
\(4n^2\) summands. Since the absolute value of every exponent is at most
\((L+1)R\), termwise differentiation gives
\(\Psi_L(y)\le4n^2e^{(L+1)R}\),
\(\|\nabla\Psi_L(y)\|_2\le4n^2(L+1)e^{(L+1)R}\), and
\(\|\nabla^2\Psi_L(y)\|_2
\le4n^2(L+1)^2e^{(L+1)R}\).
Thus the first group of upper bounds in the lemma holds, for example,
with \(P_1(n,L)=4n^2(L+1)^2\).

It remains to bound the log-convexity matrix \(M_{\Psi_L}\). For
\(g(y)=\sum_a c_a e^{a^\top y}\), direct expansion of
\(M_g=g\nabla^2g-\nabla g\nabla g^\top\) gives
\[
    M_g(y)
    =
    \frac12
    \sum_{a,b}
    c_a c_b e^{(a+b)^\top y}
    (a-b)(a-b)^\top .
\]
For each \(j\), the support of \(\Psi_L\) contains the two exponent
vectors \(sLe_i+e_j\) and \(sLe_i-e_j\). Their difference is \(2e_j\),
while their sum has inner product \(2LR\) with \(y\). The two ordered
pairs formed from these vectors therefore contribute
\(4e^{2LR}e_je_j^\top\) to \(M_{\Psi_L}(y)\). Summing over \(j\) gives
\[
    M_{\Psi_L}(y)
    \succeq
    4e^{2LR}I_n,
\]
and hence
$\lambda_{\min}\left(M_{\Psi_L}(y)\right)\geq 4e^{2LR}$.

Finally, using
\(\|M_{\Psi_L}\|_2
\le\Psi_L\|\nabla^2\Psi_L\|_2+\|\nabla\Psi_L\|_2^2\)
together with the bounds above gives
\(\|M_{\Psi_L}(y)\|_2
\le32n^4(L+1)^2e^{2(L+1)R}\).
Hence the second upper bound holds with
\(P_2(n,L)=32n^4(L+1)^2\).
\end{proof}

\subsection{Proof of Theorem~\ref{thm:global-log-coordinate}}
\label{app:global-log-coordinate-proof}

\begin{proof}
Let
\(q(y)=\sum_{|\ell|\le4}c_\ell e^{h\ell^\top y}\)
be the hard exponential form with curvature margin
\(\gamma>0\). Recall that
\(D=\max_{\alpha\in\operatorname{supp}(q)}\|\alpha\|_1\le2\).
Let
\(C_2=\sum_{|\ell|\le4}|c_\ell|\|h\ell\|_2^2\).
Then, writing \(R=\|y\|_\infty\),
\(\|\nabla^2q(y)\|_2\le C_2e^{DR}\).

We construct
\(Q_L(y)=q(y)+\delta_L\Psi_L(y)\), where
\(\Psi_L\) is the auxiliary function from
Lemma~\ref{lem:auxiliary-exponential}. We first choose \(\delta_L\) to
preserve the interior negative-curvature witness, and then choose \(L\)
so that the auxiliary curvature dominates outside the unit box.

On \([-\frac12,\frac12]^n\),
Lemma~\ref{lem:auxiliary-exponential} and the explicit bound from its
proof give
\[\|\nabla^2\Psi_L(y)\|_2\le 4n^2(L+1)^2e^{(L+1)/2}.\]
Using \(\sqrt e<5/3\), let
\[
    U_L
    =
    4n^2(L+1)^2
    \left(\frac53\right)^{L+1},
    \qquad
    \delta_L
    =
    \frac{\gamma}{2U_L}.
\]
Then
\(\delta_L\|\nabla^2\Psi_L(y)\|_2\le\gamma/2\) throughout the witness box.

In the negative alternative, let
\(\widehat y\in[-\frac12,\frac12]^n\) and a unit vector \(z\) satisfy
\(z^\top\nabla^2q(\widehat y)z\le-\gamma\). The preceding choice of
\(\delta_L\) gives
\[
\begin{aligned}
    z^\top\nabla^2Q_L(\widehat y)z
    =
    z^\top\nabla^2q(\widehat y)z
    +
    \delta_L z^\top\nabla^2\Psi_L(\widehat y)z
    \le
    -\gamma
    +
    \delta_L\|\nabla^2\Psi_L(\widehat y)\|_2
    \le
    -\frac{\gamma}{2}.
\end{aligned}
\]
Thus the strict negative-curvature witness is preserved.

Now consider the positive alternative. On \([-1,1]^n\),
\(\nabla^2q(y)\succeq\gamma I_n\), and \(\Psi_L\) is globally convex, so
\(Q_L\) is convex on the unit box. For \(R\ge1\),
\(\|\nabla^2q(y)\|_2\le C_2e^{DR}\) implies
\(\nabla^2q(y)\succeq-C_2e^{DR}I_n\), while
Lemma~\ref{lem:auxiliary-exponential} gives
\(\nabla^2\Psi_L(y)\succeq2e^{(L-1)R}I_n\). Hence
\[
    \nabla^2Q_L(y)
    \succeq
    \left(
        2\delta_Le^{(L-1)R}
        -
        C_2e^{DR}
    \right)I_n .
\]
We choose \(L\) so that the auxiliary term already dominates at
\(R=1\). Let \(L\ge4\) be the smallest integer satisfying
\[ 
\left(\frac85\right)^L \ge
    \frac{160n^2(L+1)^2C_2}{\gamma}.
\]

Such an \(L\) exists because the left-hand side grows exponentially,
whereas the dependence on \(L\) on the right is polynomial. For the
source family here, \(C_2/\gamma\) is polynomially bounded in \(n\), so
\(L=O(\log n)\).

Indeed, substituting the choices of \(U_L\) and \(\delta_L\), and using
\(8/3<e<3\), gives
\[
\begin{aligned}
    2\delta_Le^{L-1}
    =
    \frac{\gamma e^{L-1}}
    {4n^2(L+1)^2(5/3)^{L+1}} 
    =
    \frac{3\gamma}
    {20e\,n^2(L+1)^2}
    \left(\frac{3e}{5}\right)^L
    >
    \frac{\gamma}
    {20n^2(L+1)^2}
    \left(\frac85\right)^L
    \ge
    8C_2.
\end{aligned}
\]
Since \(D\le2\) and \(e^2<8\), this implies
\(2\delta_Le^{L-1}>C_2e^D\). Moreover,
\(L\ge4\) and \(D\le2\) give \(L-1-D>0\), so the ratio
\(2\delta_Le^{(L-1)R}/(C_2e^{DR})\) increases with \(R\).
As it exceeds one at \(R=1\), we obtain
\(\nabla^2Q_L(y)\succeq0\) for every \(R\ge1\).

Consequently, \(Q_L\) is globally convex in the positive alternative
and remains nonconvex in the negative alternative. Since \(Q_L\) is the
exponential form of a signomial, these alternatives are exactly global
log-coordinate convexity and nonconvexity.

Finally, \(L=O(\log n)\), so the auxiliary exponents and the
\(4n^2\) added terms satisfy the encoding bounds. Moreover,
\(\delta_L=\gamma/(2U_L)\) has polynomially bounded rational data.
Hence the reduction preserves strong NP-hardness.
\end{proof}

\subsection{Proof of Theorem~\ref{thm:global-log-log-convexity-hardness}}
\label{app:global-loglog-proof}

\begin{proof}
Let
\(q(y)=\sum_{|\ell|\le4}c_\ell e^{h\ell^\top y}\)
be the hard exponential form with curvature margin
\(\gamma>0\). Recall that
\(D=\max_{\alpha\in\operatorname{supp}(q)}\|\alpha\|_1\le2\).
We first apply the same scaling and offset used in the compact log-log
reduction. Let \(C_q=\sum_{|\ell|\le4}|c_\ell|\). Since
\(h\le\frac12\), \(|\ell|\le4\), and
\(\|y\|_\infty\le1\), we have
\(|q(y)|\le8C_q\) and
\(\|\nabla q(y)\|_2\le16C_q\) on the unit box.
Let \(B_0=1+8C_q\), \(B_1=1+16C_q\), and choose
\(\varepsilon=\min\{1/(2B_0),\,\gamma/(4B_1^2)\}\).
Set \(g_0(y)=1+\varepsilon q(y)\). Then
\(g_0\ge\frac12\) on the unit box and
\[
    M_{g_0}
    =
    \varepsilon g_0\nabla^2q
    -
    \varepsilon^2\nabla q\nabla q^\top .
\]
In the positive alternative,
\(M_{g_0}(y)\succeq(\varepsilon\gamma/4)I_n\) throughout the unit box.
In the negative alternative, at the witness \(\widehat y\),
\[
    z^\top M_{g_0}(\widehat y)z
    \le
    -\frac{\varepsilon\gamma}{2}.
\]
Let
\(\nu=\min\{1,\varepsilon\gamma/2\}\).
Thus \(g_0\) is positive and log-convex on the unit box in the positive
alternative, while in the negative alternative it has a strict
\(M_{g_0}\)-curvature witness of margin at least \(\nu\).
Since \(\varepsilon\) and \(\gamma\) have inverse-polynomial magnitude,
so does \(\nu\).

Write
\(g_0(y)=\sum_{|\ell|\le4}b_\ell e^{h\ell^\top y}\), absorbing the
constant term into \(b_0\). Let
\(K_0=\sum|b_\ell|\),
\(K_1=\sum|b_\ell|\|h\ell\|_1\), and
\(K_2=\sum|b_\ell|\|h\ell\|_1^2\).
For \(R=\|y\|_\infty\),
\(|g_0(y)|\le K_0e^{DR}\),
\(\|\nabla g_0(y)\|_2\le K_1e^{DR}\), and
\(\|\nabla^2g_0(y)\|_2\le K_2e^{DR}\).
Consequently,
\(\|M_{g_0}(y)\|_2\le K_Me^{2DR}\), where
\(K_M=K_0K_2+K_1^2\).

Now define
\(G_L(y)=g_0(y)+\delta_L\Psi_L(y)\).
For twice differentiable functions \(u,v\), let
\(\mathcal C(u,v)
=u\nabla^2v+v\nabla^2u
-\nabla u\nabla v^\top-\nabla v\nabla u^\top\).
Then
\[
    M_{G_L}
    =
    M_{g_0}
    +
    \delta_L\mathcal C(g_0,\Psi_L)
    +
    \delta_L^2M_{\Psi_L}.
\]
Using the bounds from Lemma~\ref{lem:auxiliary-exponential} gives
\[
    \|\mathcal C(g_0,\Psi_L)(y)\|_2
    \le
    K_Le^{(D+L+1)R},
\]
where
\(
K_L=
4n^2\bigl(
K_0(L+1)^2+K_2+2K_1(L+1)
\bigr)
\).

We first choose \(\delta_L\) to preserve the interior witness.
On \([-\frac12,\frac12]^n\), the bound \(D\le2\) gives
\(\|\mathcal C(g_0,\Psi_L)\|_2
\le K_Le^{(L+3)/2}\).
Using \(\sqrt e<5/3\) and \(e^{3/2}<5\), let
\(V_L=5K_L(5/3)^L\), so that the cross term is bounded by \(V_L\)
throughout the witness box. On the same box,
Lemma~\ref{lem:auxiliary-exponential} gives
\(\|M_{\Psi_L}\|_2
\le32n^4(L+1)^2e^{L+1}\).
Let
\(S_L=10n^2(L+1)(5/3)^L\); then
\(\|M_{\Psi_L}\|_2\le S_L^2\).

Set
\(T_L=\max\{1,V_L,S_L\}\) and
\(\delta_L=\nu/(8T_L)\).
Then
\(\delta_LV_L\le\nu/8\) and
\(\delta_L^2S_L^2\le\nu^2/64\le\nu/64\).
At the negative witness,
\[
\begin{aligned}
    z^\top M_{G_L}(\widehat y)z
    \le
    -\nu+\delta_LV_L+\delta_L^2S_L^2 
    \le
    -\nu+\frac{\nu}{8}+\frac{\nu}{64}
    <
    -\frac{\nu}{2}.
\end{aligned}
\]
Moreover,
\(G_L(\widehat y)>0\) because
\(g_0(\widehat y)\ge\frac12\) and \(\Psi_L>0\).
Hence
\(z^\top\nabla^2\log G_L(\widehat y)z<0\), so the negative
alternative is preserved.

It remains to choose \(L\) so that \(G_L\) is globally positive and,
in the positive alternative, globally log-convex. Let
\(P_L=1+5K_L+10n^2(L+1)\). Since
\(T_L\le P_L(5/3)^L\), we have
\(\delta_L\ge\frac{\nu}{8P_L}(3/5)^L\).

We choose \(L\ge4\) to satisfy three requirements. The first makes the
auxiliary term dominate \(g_0\) outside the unit box, the second makes
\(\delta_L^2M_{\Psi_L}\) dominate the cross term, and the third makes
it dominate \(M_{g_0}\). Concretely, choose \(L\) as the smallest
integer satisfying
\[
\begin{aligned}
    \left(\frac85\right)^L
    &\ge
    \frac{128P_LK_0}{\nu}
    &&\text{(global positivity)},\\
    \left(\frac85\right)^L
    &\ge
    \frac{84P_LK_L}{\nu}
    &&\text{(cross-term domination)},\\
    \left(\frac{64}{25}\right)^L
    &\ge
    \frac{1760P_L^2K_M}{\nu^2}
    &&\text{(source-curvature domination)}.
\end{aligned}
\]
Such an \(L\) exists because \(P_L\) and \(K_L\) grow only
polynomially in \(L\) and \(n\), whereas the left-hand sides grow
exponentially. Hence \(L=O(\log n)\).

We first verify global positivity. For \(R\ge1\), choose \(i,s\) with
\(sy_i=R\), and choose \(t\) so that \(ty_j\ge0\) for any fixed \(j\).
The corresponding summand of \(\Psi_L\) is at least \(e^{LR}\), so
\(\Psi_L(y)\ge e^{LR}\). On the other hand,
\(|g_0(y)|\le K_0e^{DR}\).
At \(R=1\), the lower bound on \(\delta_L\), the global-positivity
condition above, and \(3e/5>8/5\) give
\[
\begin{aligned}
    \delta_Le^L
    \ge
    \frac{\nu}{8P_L}
    \left(\frac{3e}{5}\right)^L 
    >
    \frac{\nu}{8P_L}
    \left(\frac85\right)^L
    \ge
    16K_0.
\end{aligned}
\]
Since \(D\le2\) and \(e^2<8\), we have
\(16K_0>K_0e^D\). Thus
\(\delta_Le^L>K_0e^D\) at \(R=1\).
Because \(L>D\), the ratio
\(\delta_Le^{LR}/(K_0e^{DR})\) increases with \(R\), so
\(\delta_Le^{LR}>K_0e^{DR}\) for every \(R\ge1\).
Hence \(G_L(y)>0\) outside the unit box. Together with
\(g_0\ge\frac12\) and \(\Psi_L>0\) on the unit box, this proves that
\(G_L\) is globally positive in both alternatives.

Now suppose the positive alternative holds. On the unit box,
\(g_0\) and \(\Psi_L\) are positive and log-convex,
so their positive sum \(G_L=g_0+\delta_L\Psi_L\) is log-convex there.

For \(R\ge1\), the bounds above and
Lemma~\ref{lem:auxiliary-exponential} give
\[
    M_{G_L}(y)
    \succeq
    \left(
        4\delta_L^2e^{2LR}
        -
        \delta_LK_Le^{(D+L+1)R}
        -
        K_Me^{2DR}
    \right)I_n.
\]

We first compare the positive auxiliary term with the cross term.
At \(R=1\),
\[
\begin{aligned}
    \frac{
        4\delta_L^2e^{2L}
    }{
        \delta_LK_Le^{D+L+1}
    }
    =
    \frac{4\delta_Le^{L-D-1}}{K_L} 
    \ge
    \frac{\nu}{2P_LK_Le^{D+1}}
    \left(\frac{3e}{5}\right)^L 
    >
    \frac{\nu}{42P_LK_L}
    \left(\frac85\right)^L
    \ge2,
\end{aligned}
\]
where \(D\le2\), \(e^3<21\), and the last step uses the
cross-term domination condition.

Similarly, comparing the positive auxiliary term with
\(M_{g_0}\) gives, at \(R=1\),
\[
\begin{aligned}
    \frac{
        4\delta_L^2e^{2L}
    }{
        K_Me^{2D}
    }
    \ge
    \frac{\nu^2}{16P_L^2K_Me^{2D}}
    \left(\frac{3e}{5}\right)^{2L} 
    >
    \frac{\nu^2}{880P_L^2K_M}
    \left(\frac{64}{25}\right)^L
    \ge2,
\end{aligned}
\]
where \(D\le2\), \(e^4<55\), and the last step uses the
source-curvature domination condition.

Thus, at \(R=1\), the positive auxiliary term is at least twice each
negative contribution. Moreover, \(L\ge4\) and \(D\le2\) imply
\(2L>D+L+1\) and \(2L>2D\), so both ratios increase with \(R\).
Therefore, for every \(R\ge1\),
\[
    4\delta_L^2e^{2LR}
    \ge
    2\delta_LK_Le^{(D+L+1)R},
    \qquad
    4\delta_L^2e^{2LR}
    \ge
    2K_Me^{2DR}.
\]
It follows that \(M_{G_L}(y)\succeq0\) throughout the exterior region.
Combined with log-convexity on the unit box and the global positivity
proved above, \(G_L\) is globally log-convex in the positive
alternative.

Thus, \(G_L\) is positive on \(\mathbb{R}^n\) in both alternatives. In
the positive alternative, \(\log G_L\) is convex on
\(\mathbb{R}^n\); in the negative alternative, the preserved interior
witness shows that \(\log G_L\) is not convex. Since \(G_L\) is the
exponential form of a signomial, the corresponding signomial satisfies
the positivity promise in both alternatives and distinguishes the two
cases of the unconstrained \LLC{} problem.

Finally, \(L=O(\log n)\), while
\(\varepsilon\), \(\nu\), and \(\delta_L\) have polynomially bounded
rational data, and only \(4n^2\) auxiliary terms are introduced.
Hence the construction preserves the encoding bounds and strong
NP-hardness.
\end{proof}

\end{document}